\documentclass[a4paper,twoside]{amsart}
\usepackage{xypic,amscd,amssymb,latexsym,srcltx}
\usepackage{dsfont}
\usepackage{amsthm}
\def\@setcopyright{}
\makeatother

\newtheorem{lemma}{Lemma}[section]
\newtheorem{proposition}[lemma]{Proposition}
\newtheorem{corollary}[lemma]{Corollary}
\newtheorem{theorem}[lemma]{Theorem}

\newtheorem{question}[lemma]{Question}

\newtheorem*{acknowledgments*}{ACKNOWLEDGMENTS}

\newtheorem*{conj*}{Conjecture}
\newtheorem*{thm*}{Main Theorem}
\newtheorem*{remark*}{Remark}
\newtheorem*{lemma*}{Lemma}
\newtheorem*{example*}{[Example]}

\begin{document}
\begin{center}
{\Large \bf Elementary abelian subgroups in some special $p$-groups
}\\

\bigskip
\end{center}

\begin{center}
Xingzhong Xu$^{1, 2}$

\end{center}

\footnotetext {$^{*}$~~Date: 05/11/2017.
}
\footnotetext {
1. Department of Mathematics, Hubei University, Wuhan, 430062, China

2. Departament de Matem$\mathrm{\grave{a}}$tiques, Universitat Aut$\mathrm{\grave{o}}$noma de Barcelona, E-08193 Bellaterra,
Spain

Xingzhong Xu's E-mail: xuxingzhong407@mat.uab.cat; xuxingzhong407@126.com

Supported by NSFC grant (11501183).
}

\title{}
\def\abstractname{\textbf{Abstract}}

\begin{abstract}\addcontentsline{toc}{section}{\bf{English Abstract}}
Let $P$ be a finite $p$-group and $p$ be an odd prime.
Let $\mathcal{A}_p(P)_{\geq2}$ be a poset
consisting of elementary abelian subgroups of rank at least 2.
If the derived subgroup $P'\cong C_p\times C_p$, then
the spheres occurring in $\mathcal{A}_p(P)_{\geq2}$ all have the
same dimension.
\hfil\break


\end{abstract}

\maketitle

\section{\bf Introduction}

Let $P$ be a finite $p$-group, and the poset of all nontrivial elementary abelian $p$-subgroup of $P$
is denoted by $\mathcal{A}_p(P)$. In \cite{BT}, Bouc and Th\'evenaz focused on subposet $\mathcal{A}_p(P)_{\geq2}$
consisting of elementary abelian subgroups of rank at least 2. And they determined the homotopy type of the complex
$\mathcal{A}_p(P)_{\geq2}$: a wedge of spheres (of possibly different dimensions). We refer the reader to
\cite{Q} for more details about $p$-subgroups complex.
After computer calculations, Bouc and Th\'evenaz posed the following questions in \cite{BT}.
\begin{question}(Bouc, Th\'evenaz) Let $P$ be a $p$-group. Do the spheres occurring in $\mathcal{A}_p(P)_{\geq2}$ all have the
same dimension if $p$ is odd?  Does one get only two consecutive dimensions if $p=2$?
\end{question}
In \cite{Bo2}, Bornand proved that Question 1.1 has a positive answer if $P$
has a cyclic derived subgroup. Recently, \cite{FS} provided a negative answer to Question 1.1.
But determining the dimensions of those spheres of $\mathcal{A}_p(P)_{\geq2}$
is still an interesting question,
 thus  the following theorem is the main result in this short paper.
\begin{theorem} Let $P$ be a finite $p$-group and $p$ be an odd prime. If the derived subgroup $P'\cong C_p\times C_p$, then
the spheres occurring in $\mathcal{A}_p(P)_{\geq2}$ all have the
same dimension.
\end{theorem}

$Structure~ of ~ the~ paper:$ After recalling the basic properties of finite $p$-group with cyclic derived group in Section 2,
we give a proof of Theorem 1.2  in Section 3.

\section{\bf Notation and finite $p$-groups}

In this section, let $p$ be an odd  prime. We refer the reader to \cite{G, Ha} for the theory of finite group and topology respectively.
Now, recall that a minimal nonabelian $p$-group means that all its proper subgroups are abelian. Let $P$ be a minimal nonabelian $p$-group. Then $P$
is one of the following groups:
\begin{eqnarray*}
& ~ & S(n,m,0):=\langle
x,y|x^{p^{n}}=y^{p^{m}}=1, [x,y]=x^{p^{n-1}}\rangle(n\geq 2);\\
& ~ & S(n,m,1):=\langle x,y,z|x^{p^{n}}=y^{p^{m}}=z^{p}=1, [x,y]=z,
[z,x]=[z,y]=1\rangle.\\
\end{eqnarray*}
Here, $S(n,1,1)\cong S(1,n,1)$.
One can find these results about nonabelian $p$-groups in \cite{B}, \cite{R}.
In particular, $S(2,1,0)$ and $S(1,1,1)$ are extraspecial $p$-groups. And they always are written
as $p_-^{1+2}, p_+^{1+2}$ respectively.

Recall that a group $H$ is said to be the central product of its normal subgroups $H_{1},H_{2}\ldots, H_{n}$ if $H=H_{1}H_{2}\cdots H_{n}$,
$[H_{i}, H_{j}]=1$ for $i\neq j$, and $H_{i}\cap \prod_{j\neq i}H_{j}\leq Z(H)$ for all $i$.

The classification of finite $p$-groups with derived subgroup of prime order, traces back to S. Blackburn.

\begin{proposition}\cite[Proposition 9]{Bl} Let $p$ be an odd  prime and $S$ be a finite $p$-group. If $|S'|=p$, then $S$ is a central product as following types:
$$S\cong S_{1}\ast S_{2}\ast\cdots\ast
S_{l}\ast A$$ where $S_{i}\cong S(n_{i},m_{i},0)~\mathrm{or}~S(n_{i},m_{i},1)$
and $A$ is a finite abelian $p$-group. Here, $S_{i}\cap (\prod_{j\neq i}S_{j})A=S'$ and $A\cap \prod_{i}S_{i}\leq S'$.
\end{proposition}

Moreover, if $Z(S)$ is cyclic, we get the following corollary which will be used to prove Theorem 1.2 in Section 3.
\begin{corollary}
Let $p$ be an odd  prime and $S$ be a finite $p$-group. If $|S'|=p$ and $Z(S)$ is cyclic, then $S$ is a central product as following types:
$$S\cong S(n,1,0)\ast p_+^{1+2}\ast\cdots\ast
p_+^{1+2},~\mathrm{or}~S\cong p_+^{1+2}\ast p_+^{1+2}\ast\cdots\ast
p_+^{1+2}.$$
\end{corollary}
%
%\begin{remark}
%\end{remark}
%
%\begin{definition}
%\end{definition}
%
%\begin{proposition}
%\end{proposition}
%
%\begin{theorem}
%\end{theorem}

\section{\bf The proof of the main theorem}

In this section, we give a proof of Theorem 1.2.

\begin{theorem} Let $P$ be a finite $p$-group and $p$ be an odd prime. If $P'\cong C_p\times C_p$, then
the spheres occurring in $\mathcal{A}_p(P)_{\geq2}$ all have the
same dimension.
\end{theorem}

\begin{proof} Suppose that the spheres occurring in $\mathcal{A}_p(P)_{\geq2}$ do not have the
same dimension, then we want to get a contradiction to prove the theorem.

Since $P'\cong C_p\times C_p$, we can set $P'=\langle a\rangle\times \langle b\rangle$ and
$|a|=|b|=p$. If $Z(P)$ is not cyclic, then $\Delta(\mathcal{A}_p(P)_{\geq2})$ is contractible. That is
 a contradiction. Hence, $Z(P)$ is cyclic and we write $Z:=\langle a\rangle$
 for the unique subgroup of order $p$ in $Z(P)$ because $Z(P)\cap P'\geq Z$.
Moreover, we can assume $b\notin Z(P)$.

{\bf Step 1.} We set $E_0=P'$ and
$$M:=C_P(E_0):=\langle a, b, y_1, y_2,\ldots, y_n\rangle$$
for some  $ y_1, y_2,\ldots, y_n \in P$.
Let $\chi:=\{E_1, E_2,\ldots, E_k\}$ be the subset of $\mathcal{A}_p(P)_{>Z}$ of all elementary
abelian subgroups of rank 2 not contained in $M$. Suppose $\chi$ is empty set, thus
$\Delta(\mathcal{A}_p(P)_{\geq2})$ is contractible. That is
 a contradiction. Hence  $\chi$ is not empty set.
Then there exists $x\in E_1-M$, it implies that
$$[x, b]\neq 1, ~\mathrm{and}~ |x|=p.$$
We assert that $[x, b]=a^i$ for some $1\leq i\leq p-1$.
If $1\neq[x,b]=ba^i$, then we can see $[x, ba^i]=ba^i$ because $a\in Z(P)$. That is a contradiction.
Hence, $[x,b]=a^i$. Therefore, we can assume that
$$[x, b]=a.$$

{\bf Step 2.} We assert that $M'\leq\langle a\rangle$. Since $M'\leq P'=\langle a, b\rangle$, for each $y_i, y_j$, we have
$$[y_i, y_j]=a^{\varepsilon}b^\epsilon.$$
where $\varepsilon,\epsilon\in \mathbb{N}$.
We can see that
$$[x, [y_i, y_j]]=[x, a^{\varepsilon}b^{\epsilon}]=a^{\epsilon}.$$
But $[[x, y_i], y_j]=1$ and $[[y_j, x], y_i]=1$ because
$[x, y_i], [y_j, x]\in P'= \langle a, b\rangle= E_0$ and $y_i, y_j\in M=C_{P}(E_0)$.
By the Hall-Witt identity, we have $[x, [y_i, y_j]]=1$. It implies $b^{\epsilon}=1$.
So $\epsilon\equiv 0 ~(\mathrm{mod}~p)$.
Hence $$[y_i, y_j]\in \langle a\rangle$$
for every $y_i, y_j$. Because $\langle a, b\rangle\leq Z(M)$, thus $M'\leq\langle a\rangle$.

By the definition of $E_i$, we can set $E_i=\langle xt_i, a\rangle$ where $t_i\in M$ such that
$$ME_i=P.$$
In other words, set $F_i=\langle xt_i\rangle$ such that
$$P=M\rtimes F_i.$$
We consider the groups $C_M(E_i)$.
We set $P_i=C_M(E_i).$
Since $P_i\leq M$, thus $P_i'\leq M'\leq \langle a\rangle$.
Therefore, we have $P_i'\cong C_p$ or $P'=1$.

{\bf Step 3.} By the proof of \cite[Thoerem 14.1]{BT}, we can see that
$$\Delta(\mathcal{A}_p(P)_{\geq2})\simeq \bigvee_{i=1}^k\Delta(\mathcal{A}_p(P_i)_{\geq2}).$$
Since the spheres occurring in $\mathcal{A}_p(P)_{\geq2}$ do not have the
same dimension, thus we can suppose that the dimension of the sphere occurring in $\mathcal{A}_p(P_1)_{\geq2}$ do not equal the
 dimension of the sphere occurring in $\mathcal{A}_p(P_2)_{\geq2}$, and
 $\Delta(\mathcal{A}_p(P_1)_{\geq2})$ and $\Delta(\mathcal{A}_p(P_2)_{\geq2})$ are not contractible.

By Step 2, we can see that $P_1'\cong P_2'\cong C_p$ and $Z(P_1), Z(P_2)$ are cyclic. Here, we can set $P_1=C_M(E_1)=C_M(\langle x, a\rangle)$
and $P_2=C_M(E_2)=C_M(\langle xt, a\rangle)$ where $1=t_1, t=t_2\in M$.
By Corollary 2.2, we can set
$$P_1\cong S_{1}\ast S_{2}\ast\cdots\ast
S_{l}\ast A$$ where $S_{i}\cong S(n_{i},m_{i},0)~\mathrm{or}~S(n_{i},m_{i},1)$.
Since $Z(P_1)$ is cyclic, we can set
$$P_1\cong S_{1}\ast S_{2}\ast\cdots\ast
S_{l}\cong\underbrace{S(n_{1},1,0)\ast p_{+}^{1+2}\ast p_{+}^{1+2}\ast\cdots\ast
p_{+}^{1+2}}_l, $$ or
$$P_1\cong S_{1}\ast S_{2}\ast\cdots\ast
S_{l}\cong \underbrace{p_{+}^{1+2}\ast p_{+}^{1+2}\ast\cdots\ast
p_{+}^{1+2}}_l.$$
If $P_1\cong S(n_{1},1,0)\ast p_{+}^{1+2}\ast p_{+}^{1+2}\ast\cdots\ast
p_{+}^{1+2},$ we can set $P_1= \langle a, u_1, u_2,\ldots, u_{2l}\rangle$
such that $$\langle a, u_{2i-1}, u_{2i}\rangle\cong S_i, ~\mathrm{and }~ u_1^{p^{n_{1}-1}}=a.$$
We can see that every elementary
abelian subgroups of rank at least 2 must contain
$\langle u_1^{p^{n_{1}-1}}, u_2\rangle$. Hence,
  $\Delta(\mathcal{A}_p(P_1)_{\geq2})$ is contractible.
Hence, $$P_1\cong S_{1}\ast S_{2}\ast\cdots\ast
S_{l}\cong \underbrace{p_{+}^{1+2}\ast p_{+}^{1+2}\ast\cdots\ast
p_{+}^{1+2}}_l.$$

Similarly, we can set
$$P_2\cong R_{1}\ast R_{2}\ast\cdots\ast
R_{s}\cong \underbrace{p_{+}^{1+2}\ast p_{+}^{1+2}\ast\cdots\ast
p_{+}^{1+2}}_s.$$
Here, $n_1, n_2, l,s\in \mathbb{N}$. And we can see
$P_1=\langle a, u_1, u_2,\ldots, u_{2l}\rangle$
such that $$\langle a, u_{2i-1}, u_{2i}\rangle\cong S_i.$$
And we also can set
$P_2=\langle a, v_1, v_2,\ldots, v_{2s}\rangle$
such that $$\langle a, v_{2i-1}, v_{2i}\rangle\cong R_i.$$

By \cite[Proposition 4.7]{Bo2}, we can see that the dimension of the sphere occurring in $\mathcal{A}_p(P_1)_{\geq2}$ is $l-1$
and the dimension of the sphere occurring in $\mathcal{A}_p(P_2)_{\geq2}$ is $s-1$.
Since $s\neq l$, we can assume that $s< l$.

Since $P_1, P_2\leq M$ and $|M'|=p$, we can assume that $\langle u_{2i-1}, u_{2i}\rangle \cap P_2\leq Z(\langle P_1, P_2\rangle )$ for some $i$ because
 $s< l$.
Hence $[xt, u_{2i-1}]\neq 1$ and $[xt, u_{2i}]\neq 1$. But $[x, u_{2i-1}]=1$ and $[x, u_{2i}]=1$ (because
$u_{2i-1}, u_{2i}\in P_1=C_M(E_1)=C_M(\langle x, a\rangle)$),
thus $[t, u_{2i-1}]\neq 1$ and $[t, u_{2i}]\neq 1$.

Moreover, we can set
$$[t, u_{2i-1}]=a^{\varepsilon_1}\neq 1,$$
$$[t, u_{2i}]=a^{\epsilon_1}\neq 1$$
for some numbers $\varepsilon_1, \epsilon_1\in \mathbb{N}$.
There exist two numbers $(\varepsilon_1', \epsilon_1')\neq (0 ,0) ~(\mathrm{mod}~ p)$ such that
$\varepsilon_1\varepsilon_1'+\epsilon_1\epsilon_1' \equiv 0~(\mathrm{mod}~p)$.
Hence $$[t, u_{2i-1}^{\varepsilon_1'}u_{2i}^{\epsilon_1'}]=a^{\varepsilon_1\varepsilon_1'+
\epsilon_1\epsilon_1'}=1.$$
Here, $u_{2i-1}^{\varepsilon_1'}u_{2i}^{\epsilon_1'}\neq 1$. It implies $1\neq u_{2i-1}^{\varepsilon_1'}u_{2i}^{\epsilon_1'}\in P_2$.
But $\langle u_{2i-1}, u_{2i}\rangle \cap P_2\leq Z(\langle P_1, P_2\rangle )$, thus $u_{2i-1}^{\varepsilon_1'}u_{2i}^{\epsilon_1'}=1$.
That is a contradiction.
\end{proof}

\textbf{ACKNOWLEDGMENTS}\hfil\break
The author would like to thank  Prof. C. Broto for his constant encouragement in Barcelona in Spain.

\end{document}